\def\mytitle{A generalized Bellman-Ford Algorithm\texorpdfstring{\\}~for Application in Symbolic Optimal Control}
\def\myname{Alexander Weber, Marcus Kreuzer and Alexander Knoll}
\def\mykeywords{Symbolic Optimal Control, Bellman-Ford Algorithm, Discrete Abstractions, High-Performance Computing, Aerial Firefighting}
\def\arxiv{0}
\documentclass[letterpaper, 10pt, conference]{ieeeconf}
\IEEEoverridecommandlockouts
\usepackage{cite}
\usepackage[english]{babel}
\usepackage{amsmath}
\usepackage{amssymb}

\usepackage{amsthm}
\usepackage{stmaryrd}
\usepackage{tikz}
\usetikzlibrary{calc,shapes,arrows,automata}
\usepackage{algorithm}
\usepackage{algpseudocode}
\algnewcommand\algorithmicinput{\textbf{Input:}}
\algnewcommand\Input{\item[\algorithmicinput]}
\swapnumbers
\newtheoremstyle{rem}{\topsep}{\topsep}{\normalfont}{0pt}{\bfseries}{.}{ }{\thmname{#1 }\thmnumber{#2}\thmnote{ \textup{(#3)}}}
\newtheorem{definition}{Definition}[section]
\newtheorem{theorem}{Theorem}[section]
\theoremstyle{rem}
\newtheorem{example}[definition]{\pushQED{\qed}Example}
\makeatletter
\def\endexample{\popQED\@endtheorem}
\def\begriff#1%
{\@nomath\begriff\ifdim\fontdimen\@ne\font>\z@%
\textbf{#1}\else\textit{#1}\fi}
\def\intcc#1{\ensuremath{\left[#1\right]}}

\makeatother
\def\powerset#1%
{\mathcal{P}(#1)%
}
\newcommand{\R}{\mathbb{R}}

\newcommand{\Frontier}{\mathcal{F}}
\def\implies{\relax\ifmmode\mathrel{\Rightarrow}\else$\implies$ \fi}
\usepackage{units}
\usepackage[decimalsymbol=.,exponent-product=\cdot]{siunitx}
\DeclareSIUnit{\rpm}{RPM}
\DeclareSIUnit{\degree}{^\circ}
\DeclareSIUnit{\vel}{\unitfrac{\meter}{\second}}
\DeclareSIUnit{\dens}{\unitfrac{\kilogram}{\meter\cubic}}
\usepackage{paralist}
\usepackage{hyperref}
\ifx\arxiv\undefined
\else
\hypersetup{
colorlinks=true,%
breaklinks=true,%
pdfdisplaydoctitle=true,%
linkcolor={black},%
citecolor={black},
urlcolor={blue},
pdfstartview={FitV},%
pdftitle={\mytitle},
pdfauthor={\myname},%
pdfsubject={Submitted to arxiv on \today.},
pdfkeywords={\mykeywords}
}
\fi
\ifx\arxiv\undefined
\else
\usepackage{fancyhdr}
\fi
\graphicspath{{figures/}} %
\title{\bf \LARGE \mytitle}
\author{
\myname
\thanks{
The authors are with the
Munich University of Applied Sciences,
Dept. of Mechanical, Automotive and Aeronautical Eng.,
80335 M\"unchen, Germany. 
}%
\thanks{This work has been supported by the German Federal Ministry of Education and Research (Project ARCUS).}%
}%
\begin{document}
\maketitle

\begin{abstract}
Symbolic controller synthesis is a fully-automated and correct-by-design synthesis scheme whose limitations are its immense memory 
and runtime requirements. 
A current trend to compensate for this downside 
is to develop techniques for parallel execution 
of the scheme both in mathematical foundation and 
in software implementation. 
In this paper we present 
a generalized Bellman-Ford algorithm to be 
used in the so-called symbolic optimal control, 
which is an extension of 
the aforementioned synthesis scheme. 
Compared to the widely used 
Dijkstra algorithm our algorithm has two advantages. 
It allows for 
cost functions taking arbitrary (e.g. negative) values 
and for
parallel execution 
with the ability for trading processing speed for 
memory consumption.
We motivate the usefulness 
of negative cost values 
on a scenario of 
aerial firefighting 
with unmanned aerial vehicles. 
In addition, this four-dimensional numerical 
example, which is rich in detail, demonstrates 
the great performance of our algorithm.
\end{abstract}
\section{Introduction}
\ifx\arxiv\undefined
\else
\thispagestyle{fancy}
\fi
Symbolic controller synthesis is attracting 
considerable attention in the last two decades, 
see 
\cite{Tabuada09,
ReissigWeberRungger17,
KhaledZamani19,
MacoveiciucReissig19} 
and the many references in these works. 
This synthesis scheme takes 
a plant and a control specification formulation 
as input and attempts to solve 
the resulting control problem algorithmically 
without an intermediate intervention 
of the engineer. 
In case the process terminates successfully, 
a controller is returned possessing 
the formal guarantee that 
the resulting closed loop 
meets the given specification.
A subsequent verification
routine is not required. 
Processable plants are sampled-data control systems 
whose underlying continuous-time dynamics are given by 
nonlinear differential equations/inclusions.
Specifications can be in principle quite arbitrary.  
However, efficient algorithms have been presented 
only for safety, reach-avoid \cite{MalerPnueliSifakis95} 
and GR(1)-specifications \cite{BloemJobstmanPitermanPnueliSaar12}. 
(Applications of these algorithms in 
symbolic control can be also found in 
\cite{Girard11,ReissigWeberRungger17,HsuMajumdarMallikSchmuck18}.)
In addition, the basic theory has been recently extended
to optimal control problems so that 
near-optimal controllers with respect 
to a given non-negative cost function can be synthesized \cite{ReissigRungger13,ReissigRungger18}.
Though making most of 
the empirical synthesis techniques obsolete in principle, 
symbolic controller synthesis 
has not become a standard technique until now. 
The main problem is the ``curse of dimensionality",  
from which this approach suffers.  
I.e. memory consumption and runtime 
are growing exponentially
with increasing state space dimension 
of the given plant. 
The runtime is due to, among other things, 
the solution of initial value problems, 
typically millions
during the first out of 
two steps of the synthesis scheme. 
Saving data generated from these solutions causes
the huge memory consumption. 
The data is then used in the second step, where aforementioned
algorithms may be used to solve an auxiliary discrete problem.
The latter steps classically execute sequentially, 
both in terms of theory and software implementation.
To raise runtime performance methods for parallel execution   
(in theory \cite{RunggerStursberg12,HsuMajumdarMallikSchmuck18,MacoveiciucReissig19} and 
in implementation \cite{KhaledZamani19}) 
have been presented recently.
More concretely, 
the pioneering work \cite{KhaledZamani19} indicates
the potential boost that can be achieved 
by utilizing high-performance computing platforms and
\cite{RunggerStursberg12,HsuMajumdarMallikSchmuck18,MacoveiciucReissig19} 
present theories for concurrent execution 
of the two steps. 

Against this backdrop, 
the contribution of this paper is twofold. 
Firstly,
we extend the class of solvable 
optimal control problems to problems, whose
cost functions take negative values. 
Moreover, the presented algorithm 
allows for efficient implementation by parallelizing both the execution of 
the algorithm itself and the two steps of the synthesis scheme described above. 
In fact, we present a version of the well-known 
Bellman-Ford algorithm \cite{Bellman58,Ford56,Moore59} for directed hypergraphs.
The Bellman-Ford algorithm (on ordinary directed graphs) 
not only applies to negative edge weights; 
In contrast to the Dijkstra algorithm \cite{Dijkstra59}
it also allows parallelization 
relatively easily \cite{BusatoBombieri16}. 
As we will show, the latter properties pass over to our novel variant.
Moreover, we present a method to regulate the memory consumption during execution in the sense that processing speed can be traded for memory consumption. 
We particularly show that our algorithm outperforms 
in a concrete example 
the memory-efficient Dijkstra-like algorithm 
recently proposed in \cite{MacoveiciucReissig19}.

We will motivate the requirement of handling negative edge weights, which arise from arbitrary cost functions, by a practical example: automated aerial firefighting with an unmanned aerial vehicle. 
The firefighting aircraft shall not only reach 
the hot spot as fast as possible
but shall be rewarded for flying over it as long as necessary 
in order to release its firefighting water. 
The existing theory on reach-avoid specifications 
is sufficient to reach the target while optimizing a non-negative cost function but insufficient to formulate a reward mechanism.
We will set up a detailed scenario of aerial firefighting
and thereby show the applicability of our theoretic results and the great performance of our algorithm.\looseness=-1

The rest of the paper includes 
the notation used for presenting our theory 
in Section \ref{s:notation}. 
In Section~\ref{s:optimalcontrol} 
we summarize the existing theory about symbolic optimal control 
such that our main contributions, 
which are included in Section~\ref{s:algorithm}, 
can be presented rigorously. 
The application of our results to sampled-data systems are discussed in Section \ref{s:application} and 
two numerical examples are included in Section~\ref{s:example}. 
A conclusion is given in Section \ref{s:conclusion}.
\section{Notation}
\label{s:notation}
The symbols
$\mathbb{R}$, $\mathbb{Z}$ and $\R_+$, $\mathbb{Z}_+$ stand for the set of real numbers, integers and non-negative reals and integers, respectively. 
The symbol $\emptyset$ denotes the empty set.
For a map $f \colon A \to \mathbb{R}$ and $c \in \mathbb{R}$ the relations $\leq$, $\geq$ hold point-wise, e.g. $f \geq c$ iff $f(a) \geq c$ for all $a \in A$. 
For $f$ and another map $g \colon A \to \mathbb{R}$ we write $f \geq g$ iff $f(a) \geq g(a)$ for all $a\in A$. 
The derivative of a map $f$ with respect to 
the first argument is denoted by $D_1f$. 
For sets $A$ and $B$ we denote 
the set of all functions $A \to B$ by $B^A$. 
A map with domain $A$ and taking values in the powerset of $B$ is denoted by $A \rightrightarrows B$. 
A set-valued map $f \colon A \rightrightarrows B$ is \begriff{strict} iff 
$f(a) \neq \emptyset$ for all $a \in A$. 
The cardinality of the set $A$ is denoted by $|A|$.
\section{The Notion of System and Optimal Control}
\label{s:optimalcontrol}
The class of optimal control problems 
that is considered in this work will 
be formalized in this section. 
To this end, 
we first introduce 
the notions of system, 
closed loop and
controller. Then we define the notion of
optimal control problem. We use herein the concepts introduced in\cite{ReissigWeberRungger17,ReissigRungger18}.

To summarize in advance, 
we investigate subsequently 
a closed loop scheme as depicted in Fig.~\ref{fig:closedloop}: 
The primary controller is to be synthesized such that 
the total cost of the evolution of the closed loop is minimized.
The total cost is obtained 
from accumulating running costs and adding
a terminal cost instantaneously 
when the primary controller hands over 
the control to a secondary controller. 
The hand-over at some finite time is 
mandatory.
(This scenario was rigorously 
defined in \cite{ReissigRungger13}.)
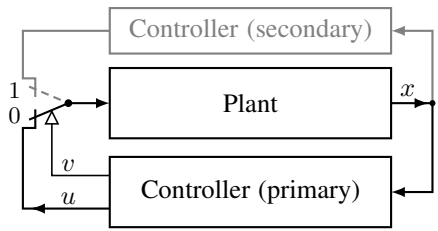
\begin{figure}
\centering
{
\newcommand{\xca}{-3.4}
\newcommand{\yca}{0}
\newcommand{\xcb}{3.4}
\newcommand{\ycb}{1.7}
\newcommand{\xcca}{\xca}
\newcommand{\ycca}{4.2}
\newcommand{\xccb}{\xcb}
\newcommand{\yccb}{5.3}
\newcommand{\xsa}{\xca}
\newcommand{\ysa}{2.15}
\newcommand{\xsb}{\xcb}
\newcommand{\ysb}{3.85}
\newcommand{\ofxa}{2.7}
\newcommand{\ofxb}{1.4}
\newcommand{\ofy}{.5}
\newcommand{\cli}{.4}
\newcommand{\sax}{\xsa-1.4}
\newcommand{\say}{\ysa/2+\ysb/2-.5}
\begin{tikzpicture}[scale=0.55,>=latex]
\draw[thick] (\xsa,\ysa) rectangle node[text width=3cm,align=center]{Plant} (\xsb,\ysb);
\draw[thick] (\xca,\yca) rectangle node[text width=3cm,align=center]{Controller (primary)} (\xcb,\ycb);
\draw[thick,gray] (\xcca,\ycca) rectangle node[text width=4cm,align=center]{Controller (secondary)} (\xccb,\yccb);
\draw[->,thick,gray] (\xsb+1,\ysa/2+\ysb/2) |- (\xccb,\ycca/2+\yccb/2) ;
\draw[->,thick] (\xsb+1,\ysa/2+\ysb/2) |- (\xcb,\yca/2+\ycb/2) ;
\draw[-,thick]  (\xsb+1,\ysa/2+\ysb/2) -- (\xsb,\ysa/2+\ysb/2);
\draw[<-,thick]  (\xsb+.9,\ysa/2+\ysb/2) -- (\xsb+.85,\ysa/2+\ysb/2);
\draw[-,semithick] (\xca,\yca/2+\ycb/2+.4) -| (\sax,\say);
\draw[-,thick]     (\xca,\yca/2+\ycb/2-.4) -| (\xsa-2.1,\ysa/2+\ysb/2-.6) -| (\xsa-1.8,\ysa/2+\ysb/2-.15);
\draw[-,thick,gray] (\xcca,\ycca/2+\yccb/2) -| (\xsa-2.1,\ysa/2+\ysb/2+.6) -| (\xsa-1.8,\ysa/2+\ysb/2+.15);
\draw[->,thick] (\xsa-1,\ysa/2+\ysb/2) -- (\xsa,\ysa/2+\ysb/2);
\draw[-,thick] (\xsa-1.95,\ysa/2+\ysb/2-\cli) -- (\xsa-1,\ysa/2+\ysb/2);
\draw[-,thick,gray,densely dashed] (\xsa-1.95,\ysa/2+\ysb/2+\cli) -- (\xsa-1,\ysa/2+\ysb/2);
\draw[semithick] (-.15+\sax,0+\say) -- (.15+\sax,0+\say) -- (0+\sax,.3+\say) -- cycle;
\draw[fill=black] (\xsa-1,\ysa/2+\ysb/2) circle (.08cm);
\draw[fill=black] (\xsb+1,\ysa/2+\ysb/2) circle (.06cm);
\node at (\xsb+.4,\ysa/2+\ysb/2+.3) {$x$};
\node at (\xca-1,\yca/2+\ycb/2+.65) {$v$};
\node at (\xca-1,\yca/2+\ycb/2-.15) {$u$};
\node at (\xsa-2.3,\ysa/2+\ysb/2-.3) {\small$0$};
\node at (\xsa-2.3,\ysa/2+\ysb/2+.3) {\small$1$};
\draw[->,thick]  (\xca-1.85,\yca/2+\ycb/2-.4) -- (\xca-1.9,\yca/2+\ycb/2-.4);
\end{tikzpicture}
}
\caption{\label{fig:closedloop}Closed loop scheme \cite{ReissigRungger13} investigated in this work.
}
\vspace*{-\baselineskip}
\end{figure}

\subsection{System and Behavior}
In this paper 
we use the following notion of system, 
which is frequently considered in literature 
in the following or similar variants, 
e.g. 
\cite{Girard10,RunggerZamani16,ReissigWeberRungger17}. 
\begin{definition}
A \begriff{system} is a triple
\begin{equation}
\label{e:def:system}
(X, U, F),
\end{equation}
where $X$ and 
$U$ are non-empty sets and 
$F \colon X \times U \rightrightarrows X$ 
is strict. 
\end{definition}
The first two components of a system $S$ in \eqref{e:def:system} 
are called \begriff{state} and \begriff{input space}, respectively. 
The third component, 
the \begriff{transition function}, 
defines a dynamic evolution for the system %
through the difference inclusion
\begin{equation}
\label{e:dynamics:discrete}
x(t+1) \in F(x(t),u(t)).
\end{equation}
For example, a dynamical system obtained by discretizing an ordinary differential equation may be formulated as a system \eqref{e:def:system} with time-discrete dynamics \eqref{e:dynamics:discrete}. (See Section \ref{ss:sampledsystem} for the details.)
The \begriff{behavior of $S$ initialized at $p \in X$}, 
which results from the imposed dynamics, is the set
\begin{equation}
\label{e:behaviour}
\{ (u,x) \in (U\times X)^{\mathbb{Z}_+} \mid p=x(0) \ \wedge \ \forall_{t\in \mathbb{Z}_+}\!\!: \eqref{e:dynamics:discrete} \text{ holds} \}.
\end{equation}
Loosely speaking, the behavior is the set of all input-output signal pairs that can be measured on the lines of the system. Subsequently, we denote \eqref{e:behaviour} by $\mathcal{B}_p(S)$. 
\subsection{Controller and Cost functional}
We investigate
the problem of synthesizing 
an optimal controller with respect to costs as we detail below.
To begin with, by a \begriff{controller} for a system $S$ of the form \eqref{e:def:system} we mean a strict set-valued map 
\[\mu \colon X \rightrightarrows U.\] 
Therefore, controllers in this work are static,
do not block, and 
do not use information from the past. 
By concept, a controller shall restrict 
the behavior of the plant to control. 
This property is 
reflected in our formalism as follows.
We define the \begriff{closed-loop behavior} 
of the controller 
$\mu$ interconnected with $S$ and 
initialized at $p \in X$ by
\begin{equation}
\mathcal{B}_p^\mu(S) = \{ (u,x) \in \mathcal{B}_p(S) \mid \forall_{t\in\mathbb{Z}_+} u(t) \in \mu(x(t))\}.
\end{equation}
Obviously, $\mathcal{B}^\mu_p(S) \subseteq \mathcal{B}_p(S)$, so this formalism is indeed compliant with intuition about controllers.

The objective that we consider 
is to minimize
the cost for operating the closed loop.
Specifically, given a 
\begriff{terminal} and 
\begriff{running cost function} of the form
\begin{subequations}
\label{e:cost}
\begin{align}
G &\colon X \to \mathbb{R} \cup \{ \infty\}, \ \text{ and } \\ 
g &\colon X \times X \times U \to \mathbb{R} \cup \{\infty\},
\end{align}
\end{subequations}
respectively, 
the controller shall minimize the cost functional
\begin{equation*}
J \colon (U \times \{0,1\} \times X)^{\mathbb{Z}_+} \to \mathbb{R} \cup \{-\infty,\infty\} 
\end{equation*}
defined as follows:
\begin{equation}
\label{e:costs}
J(u,v,x) = G(x(T)) + \sum_{t = 0}^{T-1} g(x(t),x(t+1),u(t))
\end{equation}
if 
$T := \inf v^{-1}(1) < \infty$ and 
$J(u,v,x) = \infty$ if $v = 0$. 
In words, $v$ is a boolean-valued signal 
whose first edge from $0$ to $1$ 
defines the termination time $T$ 
of the (primary) controller.
We would like to illustrate the notion of cost 
by the following example. 
It will be also continued later.
\begin{example}
\label{ex:costs}
Let $(X,U,F)$ be a system with seven states and two inputs, 
where $F$ is defined graphically 
in Fig.~\ref{fig:example}. To be specific, 
$X = \{1,2,\ldots,7\}$, $U = \{ \mathrm{b}, \mathrm{g}\}$ and
the \underline{b}lack and \underline{g}ray dashed edges, respectively, 
define the image of $F$ for the input $\mathrm{b}$ and $\mathrm{g}$, respectively. 
E.g., $F(1,\mathrm{b}) = \{1\}$, $F(4,\mathrm{g}) = \{1,2\}$. 
The running cost function $g$ is also defined graphically by the label of each edge. E.g., $g(1,1,\mathrm{b}) = 1$, $g(2,3,\mathrm{g})=-4$. 
The terminal cost function $G$ is defined by $G(x) = x$ for all $x \in X$, i.e. the label of a state in Fig.~\ref{fig:example} 
equals exactly the value of the terminal cost of the state.\\
We consider the state and input signal 
$x = (7,4,3,3,\ldots)$ and 
$u = (\mathrm{b},\mathrm{b},\ldots)$, and the termination signals 
$v_0 := (1,\ldots)$, 
$v_1 := (0,1,\ldots)$ and 
$v_2 := (0,0,1,\ldots)$. 
Then, $J(u,v_0,x) = G(7) = 7$, $J(u,v_1,x) = G(4) + g(7,4,\mathrm{b}) = 4 + 1$ and 
$J(u,v_2,x) = 3 + 1 -2 = 2$.
Analogously, for $y=(7,5,1,1,\ldots)$ it holds 
$J(u,v_1,y) = 6$, $J(u,v_2,y)=1$. 
\end{example}
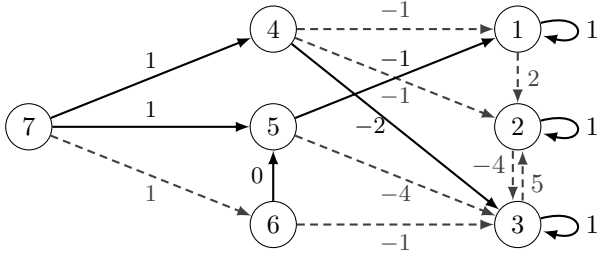
\begin{figure}
\centering
{
\newcommand{\dist}{2.5}
\begin{tikzpicture}[>=latex,scale=1.3]
\node[draw=black,circle,minimum size=.6cm] (p1) at (0,.5)  {$1$};
\node[draw=black,circle,minimum size=.6cm] (p2) at (0,-.5)  {$2$};
\node[] (p2a) at (0+.05,-.65) {};
\node[] (p2b) at (0-.05,-.65) {};
\node[draw=black,circle,minimum size=.6cm] (p3) at (0,-1.5)  {$3$};
\node[] (p3a) at (0+.05,-1.35) {};
\node[] (p3b) at (0-.05,-1.35) {};
\node[darkgray] (p3c) at (0+.2,-1.08) {$5$};
\node[darkgray] (p3d) at (0-.3,-.9) {$-4$};
\node[draw=black,fill=white,circle,minimum size=.6cm] (p4) at (-\dist,.5) {$4$};
\node[draw=black,fill=white,circle,minimum size=.6cm] (p5) at (-\dist,-.5) {$5$};
\node[draw=black,fill=white,circle,minimum size=.6cm] (p6) at (-\dist,-1.5) {$6$};
\node[draw=black,fill=white,circle,minimum size=.6cm] (p7) at (-\dist-\dist,-.5) {$7$};
\draw[->,thick] (p7) to node[above] {\small $1$} (p4) ;
\draw[->,thick] (p7) to node[above] {\small $1$} (p5) ;
\draw[->,thick,darkgray,thick,densely dashed] (p7) to node[below] {\small $1$} (p6) ;
\draw[->,thick] (p4) to node[left] {\small $-2$} (p3) ;
\draw[->,thick,darkgray,densely dashed] (p4) to node[above] {\small $-1$} (p1) ;
\draw[->,thick,darkgray,densely dashed] (p4) to node[below] {\small $-1$} (p2) ;
\draw[->,thick] (p5) to node[above] {\small $-1$} (p1) ;
\draw[->,thick,darkgray,densely dashed] (p5) to node[below] {\small $-4$} (p3) ;
\draw[->,thick] (p6) to node[left] {\small $0$} (p5) ;
\draw[->,thick,darkgray,densely dashed] (p6) to node[below] {\small $-1$} (p3) ;
\draw[->,thick,darkgray,densely dashed] (p1) to node[right] {\small $2$} (p2) ;
\draw[->,thick,darkgray,densely dashed] (p3a) -- (p2a) ;
\draw[->,thick,darkgray,densely dashed] (p2b) -- (p3b) ;
\path[]
    (p1) edge [loop right,thick] node[right] {$1$} (p1)
    (p2) edge [loop right,thick] node[right] {$1$} (p2)
    (p3) edge [loop right,thick] node[right] {$1$} (p3);
\end{tikzpicture}
\vspace*{-.5\baselineskip}
}%
\caption{\label{fig:example}System and cost functions in Examples \ref{ex:costs}, \ref{ex:closedloopperformance} and \ref{ex:algorithm}.}
\vspace*{-1\baselineskip}
\end{figure}
To a controller $\mu$ we associate a 
\begriff{closed-loop performance initialized at $p \in X$}, 
which is the value
\begin{equation}
\label{e:closedloopperformance}
L(p,\mu) = \sup_{(u,x) \in \mathcal{B}_p^\mu(S)}\inf\{ J(u,v,x) \mid v \in \{0,1\}^{\mathbb{Z}_+}\}.
\end{equation}
Roughly speaking, this quantity is the \emph{worst-case} cost for the evolution of the closed loop with controller $\mu$ for the \emph{best} possible hand-over time $T = \inf v^{-1}(1)$.
We would like to illustrate the closed-loop performance of a controller by continuing Example \ref{ex:costs}.
\begin{example}
\label{ex:closedloopperformance}
Let the system $S:=(X,U,F)$, the
cost functions $G$, $g$ and $u$, $v_2$, $x$, $y$ 
be as in Example \ref{ex:costs}. 
We consider the controller 
$\mu \colon X \rightrightarrows U$ 
defined by $\mu(x) = \{\mathrm{b}\}$ for all $x \in X$. 
Then $\mathcal{B}_7^\mu(S) = \{(u,x),(u,y)\}$.
Hence, using the results in Example \ref{ex:costs}  
we conclude that
the closed-loop performance initialized at $7$ 
satisfies $L(7,\mu) = \max\{J(u,v_2,x),J(u,v_2,y)\} = \max\{1,2\} = 2$. 
Note that a termination time greater than $2$ increases the cost by $1$ since $g(1,1,\mathrm{b}) = g(3,3,\mathrm{b}) = 1$.
\end{example}
\subsection{Optimal Control Problem}
The previously defined objects in \eqref{e:def:system} and \eqref{e:cost} 
can be grouped together in a 
compact form \cite{ReissigRungger13}. 
\begin{definition}
\label{def:ocp}
Let $S$ be a system of the form \eqref{e:def:system}.
An \begriff{optimal control problem} (on $S$) 
is a 5-tuple
\begin{equation}
\label{e:def:ocp}
(X,U,F,G,g)
\end{equation}
where $G$ and $g$ are as in \eqref{e:cost}.
\end{definition}
Finally, we relate the solution of an optimal control problem $\Pi$ of the form \eqref{e:def:ocp} to the so-called \begriff{value function}. 
The latter is the map
$V \colon X \to \mathbb{R} \cup \{-\infty,\infty\}$ defined
by
\begin{equation}
\label{e:valuefunction}
V(p) = \inf \{ L(p,\nu) \mid \nu \colon X \rightrightarrows U \text{ strict} \}.
\end{equation}
We say that 
$\mu \colon X \rightrightarrows U$ \begriff{realizes} $V$ 
if $V = L(\cdot, \mu)$. If additionally $V(p)$ is finite for all $p \in A$, where $A \subseteq X$,  
then $\mu$ is \begriff{optimal} for $\Pi$ \begriff{on} $A$. 
(Here, optimal controller and optimal termination
are formally separated, see \eqref{e:closedloopperformance}. However, as we will see later, both come naturally hand in hand.)

The focus of this work is on how to solve 
optimal control problems of the form \eqref{e:def:ocp}
algorithmically. 
Therefore, we shall review some important 
results on optimal control problems. 
In fact, we will use \cite[Th.~IV.2]{Reissig16} 
in the proofs of our main results later. 
It states that the value function is the maximal fixed point of the \begriff{dynamic programming operator}
$P \colon \intcc{-\infty,\infty}^X \to \intcc{-\infty,\infty}^X$ defined by
\begin{equation}
\label{e:dynamicprogramming}
P(W)(x) = \min \left \{ G(x), \inf_{u \in U} \sup_{y \in F(x,u)} g(x,y,u) + W(y) \right \}.
\end{equation}
The precise statement is as follows. 
\begin{theorem}
\label{t:dynamicprogramming}
Let $\Pi$ be an optimal control problem of the form \eqref{e:def:ocp} and 
let $V$ be the value function of $\Pi$ defined in \eqref{e:valuefunction}. 
Then $V$ is the maximal fixed point of the functional defined in \eqref{e:dynamicprogramming}, i.e. $P(V) = V$ and if $W\leq P(W)$ then $W\leq V$ for every $W \in \intcc{-\infty,\infty}^X$.
\end{theorem}
We would like to emphasize 
that the previous setup can be 
easily rephrased to 
the terminology of directed hypergraphs and 
the search for optimal (hyper-)paths. 
In fact, we may identify a system (with the input space being a singleton) 
with a directed hypergraph, 
where every hyperarc points to possibly several vertices \cite{AusielloDAtriSacca86}. 
This is the very reason that 
graph-theoretical algorithms can be used 
as a computational means to obtain controllers. 
Subsequently, we present such an algorithm.
\section{Main Contributions}
\label{s:algorithm}
In this section, we present our main contribution, 
which is an algorithm to determine the value function and 
the realizing controller under weaker assumptions 
than so far known on the given optimal control problem. 
Specifically, 
for the special case 
that terminal and running cost functions 
satisfy $G,g\geq 0$ a generalized 
Dijkstra algorithm was presented in \cite{ReissigRungger18} 
to solve the optimal control problem. 
(Various versions of the Dijkstra algorithm were used also in other works like \cite{GrueneJunge08,RunggerStursberg12,Reissig11,WeberReissig13}.) 
Besides the restriction to non-negative cost functions the Dijkstra algorithm has another disadvantage according to the prevailing opinion in literature: 
it cannot be conveniently parallelized due to the involved priority queue.

The novel algorithm we present 
below does not require the non-negativity of $G$ nor of $g$, and
it can be easily executed in parallel. 
In the special case of ordinary, 
directed graphs 
the novel algorithm reduces to 
the classical Bellman-Ford algorithm \cite{Bellman58}
combined with ideas of Yen \cite{Yen70} and 
Cormen et. al. \cite{CormenLeisersonRivestStein09}.
Using the techniques of the latter works
a memory and time efficient implementation can be realized.
The classical Bellman-Ford algorithm 
can be executed with a high degree of 
parallelism \cite{DavidsonBaxterGarlandOwens14,BusatoBombieri16}, 
and our algorithm inherits this property. 
We discuss implementation details in Section \ref{ss:algorithm:implementation}. 
In Section \ref{ss:algorithm}, 
we present the algorithm and its properties.
\subsection{Algorithm}
\label{ss:algorithm}
In the statement of Algorithm \ref{alg:BellmanFord} the set
\begin{equation}
\label{e:pred}
\operatorname{pred}(x,u) = \{ y \in X \mid x \in F(y,u)\}
\end{equation}
is used, which may be seen as the preimages of $F$ or, 
equivalently, as the predecessors in the hypergraph defined by $F$.
(In \eqref{e:pred}, $F$ and $X$ are as in Algorithm \ref{alg:BellmanFord}.)
\begin{algorithm}
\caption{\label{alg:BellmanFord}Generalized Bellman-Ford-Yen Algorithm}
\begin{algorithmic}[1]
\Input{Optimal control problem $(X,U,F,G,g)$}
\State{$\Frontier_1 \gets \emptyset$\hfill{}{// ``Active" Frontier \cite{CormenLeisersonRivestStein09}}}
\State{$\Frontier_2 \gets \emptyset$\hfill{}{// ``Upcoming" Frontier \cite{CormenLeisersonRivestStein09}}}
\ForAll{$x \in X$}
\State{\label{alg:BF:init:V}$W(x) \gets G(x)$}
\State{$\mu(x) \gets U$}
\If{$G(x) < \infty$}
\State{$\Frontier_1 \gets \Frontier_1 \cup ( \cup_{u \in U } \operatorname{pred}(x,u) )$}
\EndIf{}
\EndFor{\label{alg:init:endfor}}
\State{$i=0$}
\While{\label{alg:BF:while}$\Frontier_1 \neq \emptyset$ \textbf{and} $i < |X|$}
\For {\label{alg:BF:for}$(x,u) \in \Frontier_1 \times U$}
\State{\label{alg:BF:sup}$d \gets \sup_{y \in F(x,u)} g(x,y,u) + W(y) $}
\If{$d < W(x)$}
\State{\label{alg:BF:assign}$W(x) \gets d$}
\State{\label{alg:BF:control}$\mu(x) \gets \{u\}$}
\State{\label{alg:BF:F2}$\Frontier_2 \gets \Frontier_2 \cup ( \cup_{\tilde u \in U } \operatorname{pred}(x,\tilde u) )$}
\EndIf{}
\EndFor{\label{alg:BF:endfor}}
\State {$\Frontier_1 \gets \Frontier_2$\hfill{}{// Swap frontiers}}
\State{$\Frontier_2 \gets \emptyset $}
\State{$i \gets i + 1$}
\EndWhile{}
\State{\Return{$W$, $\mu$}}
\end{algorithmic}
\end{algorithm}
Algorithm~\ref{alg:BellmanFord} basically implements a fixed-point iteration
according to \eqref{e:dynamicprogramming} using additionally some heuristics to improve efficiency, which we adopt from improvements on
the classical Bellman-Ford algorithm \cite{BannisterEppstein12}. 
Firstly, Yen \cite{Yen70} observed that only a certain subset of predecessors 
need to be processed iteratively in the while-loop -- and not all elements of $X$.
Secondly, 
in \cite{CormenLeisersonRivestStein09} the two sets $\Frontier_1$ and 
$\Frontier_2$, which are called \begriff{frontiers}, 
have been introduced 
replacing the queue proposed in \cite{Yen70}.
Lastly, the second condition in line \ref{alg:BF:while} 
implements negative cycle detection. 
\begin{theorem}
\label{t:BellmanFord}
Let $S$ be a system of the form \eqref{e:def:system}, 
where $X$ and $U$ are finite. Let $\Pi$ be an optimal control problem of the form \eqref{e:def:ocp} on $S$. If Algorithm \ref{alg:BellmanFord} terminates with $\Frontier_1 = \emptyset$ then
$W$ equals the value function of $\Pi$ and $\mu$ realizes $W$. 
\end{theorem}
Before we give the proof of the theorem and discuss the time complexity of the algorithm, we would like to briefly illustrate its execution in a simple example.
\begin{example}
\label{ex:algorithm}
Let $(X,U,F,G,g)$ be the optimal control problem 
defined by means of Example \ref{ex:costs}. 
We apply Algorithm \ref{alg:BellmanFord} to it.
After initialization (line \ref{alg:init:endfor}) 
the frontier $\Frontier_1$ 
equals $X$ as every state is a predecessor
of some state.
Starting the for-loop 
in line \ref{alg:BF:for} with 
$(x,u)=(6,\mathrm{b})$ 
lines \ref{alg:BF:assign}, \ref{alg:BF:control}
imply 
$W(6)=0 + W(5) = 5$ and 
$\mu(6) = \{\mathrm{b}\}$. 
(Note that the processing order for $\Frontier_1$ is irrelevant.)
Then $\Frontier_2 = \operatorname{pred}(6,\mathrm{g}) = \{7\}$ 
in line \ref{alg:BF:F2}.
The next iteration with 
$(x,u) = (6,\mathrm{g})$ results in the changes 
$W(6) = -1 + W(3) = 2$ and 
$\mu(6) = \{\mathrm{g}\}$.  
Executing the for-loop in line \ref{alg:BF:for} again for 
$(x,u)=(4,\mathrm{g})$ yields
$W(4) = -1 + \max\{ W(1) , W(2) \} = -1 + \max\{1,2\} = 1$ as 
$F(4,\mathrm{g}) = \{1,2\}$.
Analogously, we update $W(5) = W(2) = -1$ and 
when exiting the for-loop (line \ref{alg:BF:endfor}) 
all but the state $5$ need to be processed again. 
In fact, $\Frontier_2 = X \setminus \{5\}$ 
as the $W$-values of 
$\{1\} = F(5,\mathrm{b})$ and 
$\{3\} = F(5,\mathrm{g})$ do not change.
The following execution of the while-loop (line \ref{alg:BF:while})
is therefore done with $\Frontier_1 = X \setminus \{5\}$. 
Continuing the execution as illustrated until $\Frontier_1 = \emptyset$ we finally obtain that 
the optimal controller $\mu$ satisfies $\mu(6)=\{\mathrm{b}\}$ and $\mu(s)=\{\mathrm{g}\}$ for $s \in \{2, 4, 5, 7\}$.
For states $1$ and $3$ the image of $\mu$ is $U$, which may be interpreted as the command to hand over control. 
\end{example}
\begin{proof}[Proof of Theorem \ref{t:BellmanFord}]
Denote by $V$ the value function of $\Pi$ and let $P$ be as in \eqref{e:dynamicprogramming}. 
We shall prove $W = V$. 
Assume that the inequality
$\tilde V(x) > P(\tilde V)(x)$ 
holds for some $x \in X$, where $\tilde V$ denotes the intermediate value of $W$ 
at the end of the for-loop in line \ref{alg:BF:endfor} at some iteration. 
Then there exists $u\in U$ such that 
$W(y) < \infty$ for all $y \in F(x,u)$. 
So $x \in \Frontier_2$, which implies $\tilde V(x) = P(\tilde V)(x)$ in the next iteration.
This proves $W \leq P(W)$ and therefore $W \leq V$ by Theorem \ref{t:dynamicprogramming}.
Since $W \geq P(W) \geq P(V) = V$ by lines \ref{alg:BF:init:V}, \ref{alg:BF:assign} 
and Theorem \ref{t:dynamicprogramming} 
we conclude $W \geq V$ and therefore $W = V$. 
The claim on $\mu$ is obvious.
\end{proof}
\begin{theorem}
Algorithm \ref{alg:BellmanFord} 
can be implemented 
to run with time complexity $\mathcal{O}(nm)$, 
where $n = |X|$ and 
$m = \sum_{(x,u) \in X \times U} | F(x,u) |$. 
\end{theorem}
\begin{proof}
The while-loop in Algorithm \ref{alg:BellmanFord} is executed at most $n$ times and the nested for-loop at most $m$ times, which proves the claim.
\end{proof}
\subsection{Implementation technique}
\label{ss:algorithm:implementation}
Section \ref{s:algorithm} is concluded with some notes 
on how to implement Algorithm \ref{alg:BellmanFord} efficiently.

The frontiers $\Frontier_1$ and $\Frontier_2$ 
can be realized as FIFO queues 
such that their lengths never exceed $|X|$. 
As for parallelization, 
it is readily seen that the for-loop in lines \ref{alg:BF:for}-\ref{alg:BF:endfor} can be executed in parallel where only the reading (resp. writing) operation on the array $W$ in line \ref{alg:BF:sup} (resp. line \ref{alg:BF:assign}) needs to be thread-safe. 
In this case, every thread uses its own local
frontier $\Frontier_2$ to avoid further communication among the threads.
All local frontiers are finally merged to obtain $\Frontier_1$. 

Memory consumption can be controlled quite easily. 
Firstly, only 
$\Frontier_1$, 
$\Frontier_2$, 
$W$ and 
$\mu$ need 
to reside in memory throughout execution. 
Keeping the images of
$F$ and $\operatorname{pred}$ in memory
throughout execution normally results 
in increased processing speed. 
Reading the data out of memory 
is typically faster than computing the images. 
On the other hand,
these data significantly contribute to the memory consumption.
Therefore, in case of the lack of memory, 
recomputing images of $F$ and $\operatorname{pred}$
allows to continue the execution.
Needless to say, 
the computation method for $F$ and $\operatorname{pred}$ depends
on the representation of the input data. 
\section{Application to Sampled Systems}
\label{s:application}
Within the framework of symbolic optimal control
Algorithm \ref{alg:BellmanFord} can be used
for synthesizing near-optimal controllers for 
sampled-data systems with continuous state space.
Specifically, the sampled version of a dynamical system
with continuous-time continuous-state dynamics
\begin{equation}
\label{e:ode}
\dot x(t) = f(x(t),u(t))
\end{equation}
can be considered. 
In \eqref{e:ode}, 
$f$ is a function 
$\mathbb{R}^n \times \bar U \to \mathbb{R}^n$, where
$\bar U \subseteq \mathbb{R}^m$, 
$n,m \in \mathbb{N}$. 

A brief overview of the synthesis method 
is given below
in preparation
for our experimental 
results in Section \ref{s:example}. 
More concretely, in Section \ref{ss:sampledsystem}
\begriff{sampled systems} are formalized and then
\begriff{discrete abstractions}
are introduced. 
The technique to implement a synthesis algorithm 
for control problems on sampled systems 
is outlined in Section \ref{ss:implementation}. 

A comprehensive discussion of symbolic control 
is beyond the scope 
of this paper. 
The interested reader may refer to 
\cite{ReissigWeberRungger17} for further details.
\subsection{Symbolic near-optimal control of sampled systems}
\label{ss:sampledsystem}
In this work, we assume in \eqref{e:ode} that $f(\cdot,\bar u)$ 
is locally Lipschitz-continuous 
for all $\bar u \in \bar U$ and 
$\operatorname{dom}\varphi = \mathbb{R}_+ \times \mathbb{R}^n \times \bar U$.
Here and subsequently, 
the symbol $\varphi$ is used to denote 
the general solution of \eqref{e:ode}, i.e.
$\varphi(0,x,u) = x$ and 
$D_1\varphi(t,x,u) = f( \varphi(t,x,u), u)$ 
for all $(t,x,u) \in \operatorname{dom}\varphi$.

We may formulate the discretization of \eqref{e:ode} 
with respect to a chosen sampling time 
$\tau$ as below \cite{ReissigWeberRungger17}. 
\begin{definition}
\label{def:sampledsystem}
Let $\tau > 0$. 
A system $S$ of the form \eqref{e:def:system} with 
$X = \mathbb{R}^n$, 
$U = \bar U$ 
is called \begriff{sampled system} associated with $f$ and 
$\tau$ if 
$F(x,u) = \{ \varphi(\tau,x,u) \}$ 
for all $(x, u) \in X \times U$. 
\end{definition}
Given an optimal control problem $\Pi$ of the form \eqref{e:def:ocp},
where $S=(X,U,F)$ is a sampled system,
the theory developed in \cite{ReissigRungger18} implies 
a method to be outlined below 
to compute near-optimal controllers.
Here, we say a controller $\mu$ for $\Pi$ 
is \begriff{near-optimal on $A \subseteq X$}, 
if its closed-loop performance 
\eqref{e:closedloopperformance} is finite on $A$,
i.e. $L(p,\mu) < \infty$ for all $p \in A$. 
The term ``near-optimal" can be indeed 
justified since according to the theory of 
symbolic optimal control \cite{ReissigRungger18} 
the considered synthesis method implies 
a sequence of functions
converging to the 
value function of the problem. 
It is beyond the scope of this paper 
to investigate the convergence properties
in detail.

In the first step of the aforementioned method
a certain auxiliary optimal control problem 
$\Pi' = (X',U',F',G',g')$ is defined.
By virtue of its special properties, 
which we explain later, 
theory ensures that
a near-optimal controller 
for the actual control problem is found
\emph{if} $\Pi'$ can be solved. 
The latter statement is the 
key point of this synthesis method:
$\Pi'$ is chosen having 
\emph{discrete} problem data and is solved in the second step. 
So, algorithms, as the one we have presented, 
can be used to eventually obtain a controller 
for the actual control problem.
The structure of the
controller for the actual problem 
is explained after the discussion of $\Pi'$.

The key object in $\Pi'$ is 
the so-called discrete abstraction.
A discrete abstraction is, roughly speaking, 
an approximation of the system dynamics: 
The continuous state space of the given sampled system
is quantized by means of a cover, 
a subset of the input space is picked and 
transitions between the elements of the cover
for the few chosen inputs ``cover" the transition
function of the sampled system. 
The precise definition of a discrete abstraction is the following \cite{ReissigWeberRungger17}.
\begin{definition}
\label{def:abstraction}
Let $S$ be a sampled system and 
let the system $S' = (X',U',F')$ satisfy:
\begin{enumerate}
\item $X'$ is a cover of $X$ by non-empty sets;
\item $U' \subseteq U$;
\item 
\label{def:abstraction:frr}
Let $\Omega_i \in X$, $i \in \{1,2\}$, $u \in U'$. 
If $\Omega_2 \cap F(\Omega_1,u) \neq \emptyset$ then $\Omega_2 \in F'(\Omega_1,u)$.
\end{enumerate}
Then $S'$ is called a \begriff{discrete abstraction} of $S$. 
\end{definition}

The correctness of previous method is based on the following ``overapproximation" properties of $\Pi'$:
\begin{enumerate}
\item $G(x) \leq G'(\Omega)$, whenever $x \in \Omega \in X'$;
\item $g(x,x',u) \leq g'(\Omega,\Omega',u)$, whenever $u \in U'$, $(x,x') \in (\Omega,\Omega') \in X' \times X'$;
\item The system $(X',U',F')$ is a discrete abstraction of $S$.
\end{enumerate}
The structure of the controller $\mu$ for $S$ is simple: 
It is composed of the optimal controller for $\Pi'$ and 
the quantizer induced by the cover $X'$. 
See Fig.~\ref{fig:controller}. 
\begin{figure}
\centering
\begin{tikzpicture}
\draw[thick]  (-4.5,3.5) rectangle node[text width=3cm,align=center]{$\mu \colon X' \rightrightarrows U'$} (-2,2.75);
\draw[thick]  (-1.25,3.5) rectangle node[text width=3cm,align=center]{$Q \colon\!X\!\rightrightarrows\!X'$} (0.8082,2.7733);
\draw[thick,dashed] (-4.875,2.5) rectangle (1.25,3.75);
\draw[thick,-latex] (-1.25,3.125) -- (-2,3.125);
\draw[very thick,-latex] (-4.5,3.125) -- (-5.475,3.125) -- (-5.475,3.75);
\draw[very thick,-latex]  (1.75,3.625) -- (1.75,3.125) -- (0.8082,3.125);
\node at (-5.825,3.625) {$u$};
\node at (2,3.625) {$x$};
\end{tikzpicture}
\caption{\label{fig:controller}Structure of the controller $X \rightrightarrows U$ 
for optimal control problem $\Pi$ in Section~\ref{ss:sampledsystem}. The controller is the composition of $\mu$ and $Q$, where
$\mu$ is an optimal controller for the auxiliary problem $\Pi'$ and the quantizer $Q$ is given through the property $\Omega \in Q(x) \Leftrightarrow x \in \Omega$.}
\end{figure}
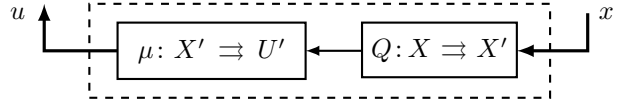
\subsection{Implementation of Symbolic Optimal Control}
\label{ss:implementation}
In order to solve optimal control problems in practice 
sophisticated choices for the ingredients 
of the discrete abstraction are required. 
In this paper, we implement 
the method of \cite{ReissigWeberRungger17,Weber18} 
in the following sense.
Let $S$ be the sampled system associated with $f$ and $\tau$. 
A discrete abstraction $(X',U',F')$ of $S$ 
is constructed as follows.
The set $X'$
consists of a finite number of compact hyper-rectangles and 
a finite number of unbounded sets such that $X'$ is a cover of $X$.
The compact sets in $X'$ are translated copies of 
\begin{equation*}
\intcc{0,\eta_1} \times \dots \times \intcc{0,\eta_n}
\end{equation*}
for some parameter $\eta \in \mathbb{R}^n_+$. 
Typically, 
these compact elements cover the ``operating range" 
of the controller to synthesize whereas the other elements
catch ``overflows".
Some finite subset $U' \subseteq U$ is chosen. 
Condition \ref{def:abstraction:frr}) in Definition \ref{def:abstraction}
is realized by overapproximating the attainable sets
$\varphi(\tau,\Omega, \bar u)$ for 
$(\Omega,\bar u) \in X'\times U'$
by hyper-rectangles \cite{KapelaZgliczynski09}.
Applying Algorithm \ref{alg:BellmanFord} 
to discrete abstractions is straightforward 
by observing that 
the algorithm remains correct if 
$\operatorname{pred}$ 
is replaced by a superset of it.
Such a set can be obtained for a sampled system 
by integrating the time-reverse dynamics 
\begin{equation}
\dot x(t)  = -f(x(t),u(t))
\end{equation}
and 
overapproximating attainable sets accordingly \cite{MacoveiciucReissig19}. 
\section{Examples}
\label{s:example}
We discuss a detailed problem about aerial firefighting 
in Section \ref{ss:firefighting} and compare the performance 
of our algorithm to the Dijkstra-like algorthm of \cite{MacoveiciucReissig19} in Section \ref{ss:landing}.
\subsection{Automated aerial firefighting}
\label{ss:firefighting}
Until now, only few scientific works discuss automated aerial firefighting.
The works \cite{ImdoukhEtAl17, QinEtAl16} focus on quadcopters and 
their properties for firefighting. 
The work \cite{HarikumarSenthilnathSundaram18} discusses 
a leader-follower firefighting strategy based on engineering experience.
At the same time, 
automated aerial firefighting would drastically increase efficiency
of the firefighting task 
while reducing risk for humans.
The reasons are divers: One main difficulty 
in extinguishing wildfires 
is that aerial firefighting 
in darkness is typically not possible 
for the sake of pilots' safety. 
So a considerable amount of time is lost.
In any case, 
firefighting pilots take a huge risk 
in those operations due to turbulence, 
smoke, other participating vehicles and the like. 
Therefore, firefighting unmanned aerial vehicles 
in combination with 
a sophisticated firefighting strategy 
would reduce risks for humans. 

As one next step towards automated 
aerial firefighting
we apply in this section 
our novel results of Section \ref{s:algorithm} 
to a simplified scenario of a wildfire:
An aircraft fights the fire
by releasing its water tank over the hot spot. 
It is not our purpose 
to present a fully realistic scenario 
that includes real problem data.
Nevertheless the presented scenario is scalable to real data 
and the full potential of symbolic controller synthesis
is also not exploited for the sake of a clear presentation.
For example,
uncertainties due to wind or
sensor noise could be 
easily taken into account 
\cite[Sect.~VI-B]{ReissigWeberRungger17} 
but would add 
some extra notation.
\subsubsection{Problem definition}
We consider a fixed-wing aircraft, 
where we model 
a) the planar motion of the aircraft, 
b) instantaneous weight loss due to water release.
In fact, we consider equations of motion given by \eqref{e:ode} 
with $f_1$ (``empty water tank"), respectively $f_2$ (``filled water tank"), 
in place of $f$. 
Here,
$f_\sigma \colon \mathbb{R}^4 \times \bar U \to \mathbb{R}^4$ with 
$\sigma \in \{1, 2\}$ is given by
\begin{equation}
\label{e:aircraft}
f_\sigma(x,u) = \begin{pmatrix}
x_4 \cdot \cos(x_3) \\
x_4 \cdot \sin(x_3) \\
m_\sigma^{-1} \cdot p_L \cdot x_4 \cdot \sin(u_2) \\
m_\sigma^{-1}(u_1 - p_D \cdot x_4^2)
\end{pmatrix},
\end{equation}
where $\bar U$ and the symbols in \eqref{e:aircraft} are explained in Tab.~\ref{tab:aircraft}. 
This point-mass model for a fixed-wing aircraft is 
widely used in literature, e.g. \cite[Sect.~3]{GloverLygeros04}.
\begin{figure*}
\centering
\input{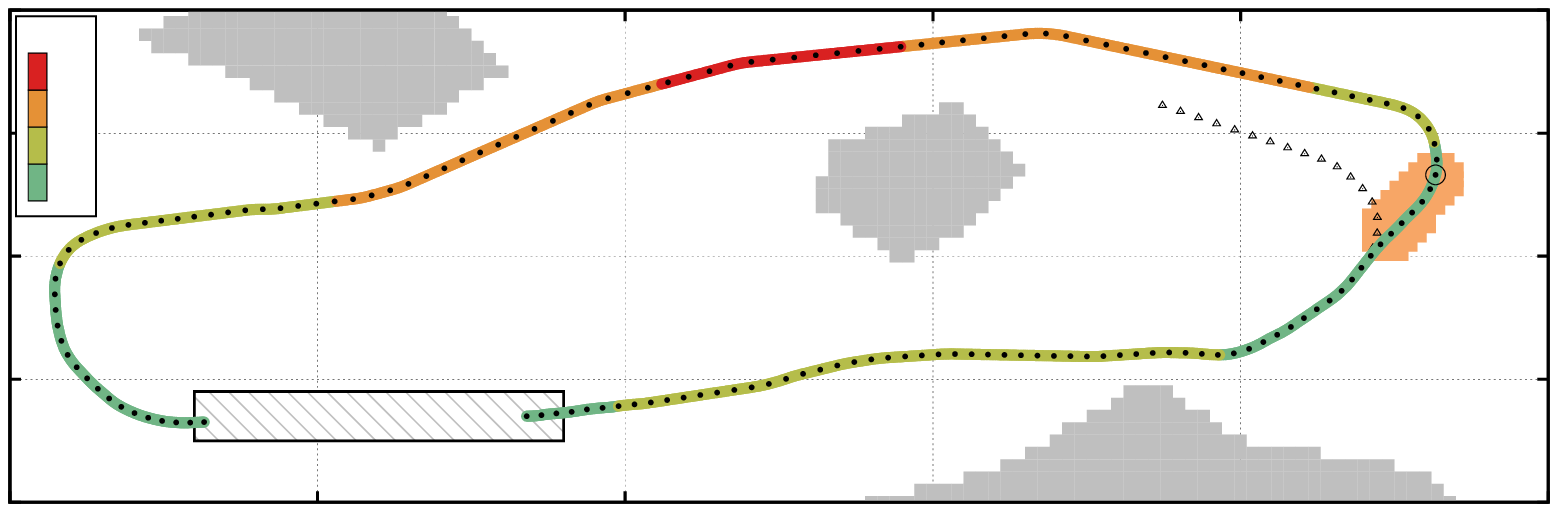}
\vspace*{-3.5em}
\caption{\label{fig:scenario}Aerial firefighting scenario from Section \ref{ss:firefighting}. A closed-loop trajectory is illustrated, which starts at 
$p_1 = (840,140,0^\circ,53)$ 
and ends at 
$p_3 \approx (319, 130, 2.43^\circ, 50.4)$. 
At the point $p_2$ ($\odot$) the controller for $\Pi_2$ hands over the control to the controller for $\Pi_1$. 
The trajectory segment indicated by the triangles belongs to a trajectory that would result if for both $\Pi_1$ and $\Pi_2$ the running cost be equal to $g_1$, i.e. without reward mechanism.
}
\end{figure*}%
\begin{table}
\centering
\begin{tabular}{|cl|}
\hline 
Symbol & Meaning and numerical values where applicable \\
\hline
\hline
$(x_1,x_2)$ & Planar position of aircraft \\
$x_3$ & Heading of aircraft \\
$x_4$ & Velocity of aircraft \\
\hline
$u_1$ & Thrust of aircraft \\
$u_2$ & Bank angle of aircraft \\
$\bar U$ & Admissible range of thrust and bank angle; \\
& $\bar U = \intcc{0, 18\!\cdot\!10^3} \times \intcc{-40^\circ , 40^\circ}$ \\
\hline
$m_1$ $[m_2]$ & Mass of aircraft without [with] its payload; \\
& $m_1 = 4250$, $m_2 = 6250$  \\
$p_D$ & Coefficient (rel. to drag) in aircraft dynamics; $p_D = 1.8$\\
$p_L$ & Coefficient (rel. to lift) in aircraft dynamics; $p_L = 85$\\
\hline
\end{tabular}
\caption{\label{tab:aircraft} Symbols used in \eqref{e:aircraft}. }
\vspace*{-1.5\baselineskip}
\end{table}%

The specification is to fly the aircraft to the fire 
after departure from the airfield, 
fly over the fire for some ``optimal" time and 
fly back to the airfield. 
Obstacles in the area must be avoided and 
the aircraft must be flown within allowed limits.
The relevant sets of the scenario
are given in Tab.~\ref{tab:specification}. 
See also Fig.~\ref{fig:scenario}.
\begin{table}
\centering
\begin{tabular}{|lll|}
\hline 
Symbol & Value & Meaning\\
\hline \hline
$X_\textrm{scen}$ & $\intcc{0,2500} \times \intcc{0,800}$ & Spatial mission area \\
$\bar X$ & $X_\textrm{scen} \times \mathbb{R} \times \intcc{50,85}$ & Operating range of controllers \\ \hline
$A_\textrm{rwy}$ & $\intcc{300,900} \times\intcc{100,180}$ & Runway of the airfield  \\
$A_\textrm{land}$ & $\intcc{-10^\circ,10^\circ} \times \intcc{50,55}$ & Admissible heading and ve- \\
& & locity of aircraft for landing \\
$A_\textrm{fire}$ & $\subseteq \mathbb{R}^2$, see Fig.~\ref{fig:scenario} & Water release region (fire)  \\
$A_\textrm{drop}$ & $\R \times \intcc{53,56}$ & Admissible heading and veloc- \\
& & ity of aircraft for water release \\
\hline
$A_\textrm{nofly}$ & $\intcc{320,880} \times\intcc{120,160}$ & Illegal aircraft states \\ & $\times \intcc{12^\circ, 348^\circ} \times \R$ & over runway\\
$A_\textrm{hill}$ & $\subseteq \mathbb{R}^2$, see Fig.~\ref{fig:scenario} & Spatial obstacle set \\
$A_\mathrm{a}$ & $A_\textrm{nofly} \cup (A_\textrm{hill} \times \R^2)$ & Overall obstacle set \\ 
\hline
\end{tabular}
\caption{\label{tab:specification} Sets defining the scenario. }
\vspace*{-2\baselineskip}
\end{table}%

We formalize this mission by two optimal control problems
$\Pi_i = (X,U,F_i,G_i,g_i)$, $i \in \{1,2\}$
which we are going to solve in succession. 
The optimal control problem $\Pi_1$
is the control task to fly from the fire back to 
the airfield with empty water tank. 
In particular, $(X, U, F_1)$ is the sampled system 
associated with $f_1$ and sampling time $\tau = 0.45$, and 
\begin{align*}
g_1(x,y,u) &= \begin{cases}
\infty , &  \textrm{if } y \in (\R^4 \setminus \bar X ) \cup A_\textrm{a} \\
\tau + u_2^2, & \text{otherwise}
\end{cases} \\
G_1(x) &= \begin{cases} 
\infty, & \textrm{if } x \notin A_\textrm{rwy} \times A_\textrm{land} \\
0, & \textrm{otherwise}
\end{cases}
\end{align*}
Loosely speaking, the optimal controller for $\Pi_1$ 
minimizes the time to arrive at the airfield 
but additionally favors small bank angles\footnote{In the definition of $g_1$, angles are taken in radians in the range $\intcc{-\pi,\pi}$.} and avoids obstacles. 
The controller terminates its action 
only if the aircraft is on ``final approach", which is enforced by $G_1$.

Optimal control problem $\Pi_2$
formalizes the control task 
to fly to the hot spot with filled water tank. 
In particular, 
$(X, U, F_2)$ is the sampled system associated with $f_2$ and $\tau$, 
\begin{equation*}
g_2(x,y,u) = \begin{cases}
-5\tau, & \textrm{if } y \in A_\textrm{fire} \times A_\textrm{drop} \\
g_1(x,y,u), & \textrm{otherwise}
\end{cases}
\end{equation*}
and $G_2 = V_1$, where $V_1$ is the value function of $\Pi_1$. 
The motivation for $g_2$ is to reward the aircraft for flying over the fire. The controller terminates its action at a beneficial state for flying back to the airfield. The latter behavior is due to the definition of $G_2$. 
\subsubsection{Auxiliary optimal control problems}
To solve $\Pi_1$ and $\Pi_2$
two auxiliary problems 
$\Pi_i' = (X',U',F'_i,G_i',g_i')$, 
$i \in \{1,2\}$ are defined 
in accordance with Section \ref{s:application}, 
where the ingredients are as follows. 
The ``abstract" state and input space 
satisfy 
$|X'| \approx 141.7 \cdot 10^6 $ and 
$|U'| = 35$, respectively. 
They are defined through subdividing the 
components of $X_\textrm{scen} \times \intcc{-\pi,\pi} \times \intcc{50,85}$ and $\bar U$
into $200\!\cdot\!70\!\cdot\!75\cdot\!135$ and $4\cdot6$ respectively, 
compact hyper-rectangles.
The running costs 
are defined by
\begin{align*}
g_1'(\Omega,\Omega',u) &= 
\begin{cases}
\infty, & \textrm{if } \Omega' \cap ( (\R^4 \setminus \bar X ) \cup A_\mathrm{a} ) \neq \emptyset \\
\tau + u_2^2, & \textrm{otherwise}
\end{cases} \\
g_2'(\Omega,\Omega',u) &= 
\begin{cases}
-5\tau, & \textrm{if } \Omega' \subseteq A_\textrm{fire} \times A_\textrm{drop} \\
g_1'(\Omega,\Omega',u), & \textrm{otherwise}
\end{cases}
\end{align*}
and the terminal costs by $G'_2 = V'_1$,
\begin{align*}
G_1'(\Omega) &= 
\begin{cases}
0, & \textrm{if } \Omega \subseteq A_\textrm{rwy}\times A_\textrm{land}  \\
\infty, & \textrm{otherwise}
\end{cases}
\end{align*}
where $V'_1$ is the value function of $\Pi_1'$.
\subsubsection{Experimental results}
\label{ss:experimentalresults}
Both $\Pi_1'$ and $\Pi_2'$ can be solved 
utilizing Algorithm \ref{alg:BellmanFord}
and near-optimal controllers for $\Pi_1$ on 
$A_\textrm{fire}\times A_\textrm{drop}$ and 
for $\Pi_2$ on 
$A_\textrm{rwy} \times A_\textrm{land}$ are found. 
A trajectory of the resulting closed loop is illustrated in Fig.~\ref{fig:scenario}. 
It is also important to note that the reward mechanism implemented by means of $g_2$ is indeed relevant to fly the aircraft a longer time over the fire. 
Disabling this mechanism, i.e. 
defining $\Pi_2$ with $g_1$ in place of $g_2$, results 
in the trajectory outlined by the triangles 
in Fig.~\ref{fig:scenario}. Hence, the hand-over command would follow 
immediately after reaching $A_\textrm{fire}$.

The performance of Algorithm \ref{alg:BellmanFord} 
shall be discussed by means of Fig.~\ref{fig:experimentalresults}.
The corresponding implementation is written in C and compiled for Linux.

Firstly,  Fig.~\ref{fig:experimentalresults}a indicates that the load 
of the random access memory can be regulated almost without restrictions.
Specifically, 
the used implementation stores 
all computed images of $F$ and $\operatorname{pred}$
until $68\%$ of RAM is consumed (predefined by user). 
After that, after every iteration (line \ref{alg:BF:endfor} in Algorithm \ref{alg:BellmanFord})
all transitions are deleted from memory
except those required in the next iteration. 
A temporary small loss of processing speed 
can be detected but the computation proceeds efficiently.
Fig.~\ref{fig:experimentalresults}b illustrates 
the great scaling property of the algorithm (and its implementation)
with respect to parallelization. 
\subsection{Aircraft landing maneuver}
\label{ss:landing}
We would like to compare the performance of
Algorithm~\ref{alg:BellmanFord} with the one of a recent, 
memory-efficient Dijkstra-like algorithm. 
To be specific, we apply our algorithm 
to the example considered in \cite[Sect.~V.B]{MacoveiciucReissig19}, 
which is a control problem about landing an aircraft DC-9 
with 3-dimensional dynamics. 
Algorithm~\ref{alg:BellmanFord} of this paper
needs 122 MB memory, which is 61\% less than the consumption
reported for \mbox{``Algorithm 2"} of \cite{MacoveiciucReissig19} (317 MB).
With one thread (resp. two threads) the computation terminates successfully in 
321 (resp. 255) seconds\footnote{Implementation and hardware as in Section \ref{ss:firefighting}.}. The authors of \cite{MacoveiciucReissig19}
report 320 seconds runtime, where only sequential computation is feasible.
\section{Conclusions}
\label{s:conclusion}
Our experimental results confirm the conclusions  
in \cite{KhaledZamani19} that 
the efficiency of symbolic controller synthesis 
can be drastically increased by utilizing algorithms
that can be executed in parallel. 
In addition, our second numerical example 
reveals that though Dijkstra-like algorithms
have a smaller time complexity
than our algorithm 
the ability of parallelization and limiting RAM usage
makes our algorithm more suitable for solving
complex control problems.
The aerial firefighting problem in Section \ref{ss:firefighting}
could not be solved, or could not be solved in a 
reasonable time, without the said properties.

In future work further techniques for 
an even higher degree of parallelization 
will be investigated, 
e.g. using graphical 
processing units. 
On such an architecture 
thousands of elementary operations 
can be executed in parallel. 
The challenge is to organize shared memory operations properly 
since they take a significant amount 
of runtime.
\begin{figure*}
\centering
\begin{minipage}{.64\textwidth}
\centering
\input{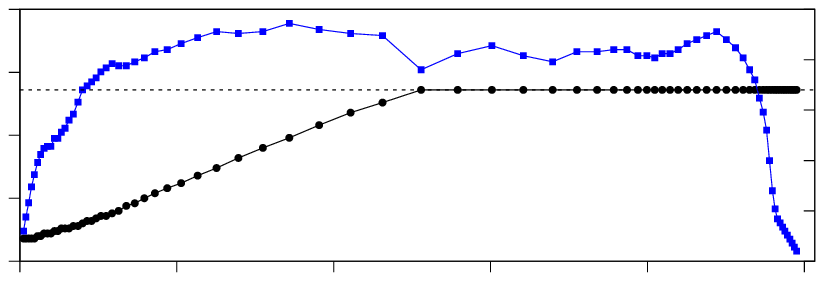}
\\
(a)
\end{minipage}%
\begin{minipage}{.35\textwidth}
\centering
\input{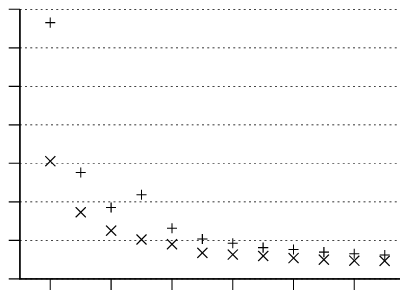}\\
(b)
\end{minipage}
\caption{\label{fig:experimentalresults}Performance analysis of Algorithm \ref{alg:BellmanFord} and its implementation applied to the optimal control problems in Section \ref{ss:firefighting}. All computations ran on up to 24 CPUs of type Intel Xeon E5-2697 v3 (2.60GHz) sharing 64 GB RAM. (a) Relation between memory use ($\bullet$) and processing speed ({\tiny{$\blacksquare$}}) when solving $\Pi'_2$ using 24 threads. 
The quantities are measured whenever the for-loop in Algorithm \ref{alg:BellmanFord} is exited (line \ref{alg:BF:endfor}). For RAM usage the system function
\texttt{getrusage} \cite{LinuxProgrammersManual} is used.
(b) Runtime in dependence of the number of threads that are used to solve problem 
$\Pi_1'$ ($\times$) and $\Pi_2'$ ($+$), respectively.}
\vspace*{-.8\baselineskip}
\end{figure*}
\section*{Acknowledgment}
The authors gratefully acknowledge the compute resources provided by the Leibniz Supercomputing Centre.
\bibliography{}
\end{document}